\documentclass[12pt,reqno]{article}
mm \textheight=8.9in

\usepackage{amssymb}
\usepackage{amsmath}
\usepackage{amsthm}
\usepackage{color}

\newcommand{\br}[3]{{$#1$}$\lower4pt\hbox{$\tp\atop\raise4pt \hbox{$\scriptscriptstyle{#2}$}$} ${$#3$}}
\newcommand{\tw}[3]{{$#1$}${\,\scriptscriptstyle {#2}}\atop\raise9pt\hbox{$\scriptstyle\tp$} ${$#3$}}
\newcommand{\ttps}[2]{{#1}\raise5pt\hbox{$\lower12pt\hbox{$\scriptstyle\tp$}\atop \lower0pt\hbox{$\tilde\;$}$}\raise4.5pt\hbox{${\scriptstyle{#2}}$}}
\newcommand{\st}[1]{\mbox{${\,\scriptscriptstyle {#1}}\atop\raise5.5pt\hbox{$*$}$}}

\newcommand{\rd}[1]{\mbox{${\,\scriptscriptstyle {#1}}\atop\raise5.5pt\hbox{$\bullet$}$}}
\newcommand{\rt}[1]{\otimes_\chi}
\newcommand{\lt}[1]{\mbox{${\,\scriptscriptstyle {#1}}\atop\raise5.5pt\hbox{$\ltimes$}$}}
\newcommand{\btr}{\raise1.2pt\hbox{$\scriptstyle\blacktriangleright$}\hspace{2pt}}
\newcommand{\btl}{\raise1.2pt\hbox{$\scriptstyle\blacktriangleleft$}\hspace{2pt}}

\newcommand{\lcr}{\raise1.0pt \hbox{${\scriptstyle\rightharpoonup}$}}
\newcommand{\rcr}{\raise1.0pt \hbox{${\scriptstyle\leftharpoonup}$}}

\newcommand{\ttp}{{\lower12pt\hbox{$\tp$}\atop \hbox{$\tilde\;$}}}

\newcommand{\id}{\mathrm{id}}

\newcommand{\Bc}{\mathcal{B}}

\newcommand{\Ac}{{\mathcal{A}}}

\newcommand{\Sc}{\mathcal{S}}

\newcommand{\J}{J}

\newcommand{\Ru}{\mathcal{R}}

\newcommand{\Fc}{\mathcal{F}}

\newcommand{\Wc}{\mathcal{W}}

\newcommand{\C}{\mathbb{C}}
\newcommand{\Sbb}{\mathbb{S}}

\newcommand{\Z}{\mathbb{Z}}

\newcommand{\N}{\mathbb{N}}

\newcommand{\tp}{\otimes}

\newcommand{\zt}{\zeta}

\newcommand{\ve}{\varepsilon}
\newcommand{\gm}{\gamma}

\newcommand{\dt}{\delta}

\newcommand{\op}{\oplus}
\newcommand{\la}{\lambda}
\newcommand{\tr}{\triangleright}

\newcommand{\End}{\mathrm{End}}

\newcommand{\Span}{\mathrm{Span}}

\newcommand{\Hom}{\mathrm{Hom}}

\newcommand{\Rm}{\mathrm{R}}

\newcommand{\La}{\Lambda}

\newcommand{\g}{\mathfrak{g}}
\renewcommand{\b}{\mathfrak{b}}
\renewcommand{\k}{\mathfrak{k}}
\newcommand{\h}{\mathfrak{h}}

\newcommand{\s}{\mathfrak{s}}

\renewcommand{\o}{\mathfrak{o}}

\newcommand{\eps}{\epsilon}

\renewcommand{\l}{\mathfrak{l}}
\renewcommand{\c}{\mathfrak{c}}

\newcommand{\si}{\sigma}
\newcommand{\al}{\alpha}
\newcommand{\bt}{\beta}

\newcommand{\be}{\begin{eqnarray}}
\newcommand{\ee}{\end{eqnarray}}

\newtheorem{thm}{Theorem}[section]
\newtheorem{propn}[thm]{Proposition}
\newtheorem{lemma}[thm]{Lemma}
\newtheorem{corollary}[thm]{Corollary}

\newcount\prg

\newcommand{\parag}{\advance\prg by1 {\noindent\bf\thesection.\the\prg\hspace{6pt}}}

\begin{document}

\title{Contravariant forms and extremal projectors}
\author{
Andrey Mudrov \hspace{1pt}\vspace{20pt}\\
Department of Mathematics,\\ \small University of Leicester, \\
\small University Road,
LE1 7RH Leicester, UK\\[10pt]
\small e-mail: am405@le.ac.uk\\
}

\date{ }

\maketitle

\begin{abstract}
Tensor product of irreducible
modules of  highest weight over a semi-simple quantum group is semi-simple  if and only if a natural
contravariant form is non-degenerate when restricted to  the span of singular vectors.
We express this restriction through the extremal projector of the quantum group providing a
computationally feasible criterion for complete reducibility of tensor products.

\end{abstract}
\\
{\small \underline{AMS classification codes}: 17B10, 17B37.}
\newpage
\section{Introduction}
This is a contituation of  \cite{M1} where we gave a complete reducibility
criterion for a tensor product $V\tp Z$ of two irreducible modules of highest weight over the
(classical or) quantum universal enveloping algebra $U_q(\g)$ of a  semi-simple Lie algebra $\g$. It is formulated in terms of  a
contravariant symmetric bilinear form  on $V\tp Z$, which is the product of the contravariant forms on the tensor factors.
Specifically, $V\tp Z$ is completely reducible if and only if the form is non-degenerate when restricted
to the span  $(V\tp Z)^+\subset V\tp Z$ of singular vectors.
In this paper, we develop an efficient  computational method for practical use of that criterion. It reveals a close relation
of the form with the extremal projector \cite{AST,KT}, which was pointed out for some special cases in \cite{M1}.

We employ a parametrization of $(V\tp Z)^+$ by a certain subspace in one of the tensor factors, e.g. $V^+_Z\subset V$.
It is isomorphic to $\Hom_{U_q(\g_+)}(Z^*, V)$,
where $U_q(\g_+)$ denotes the positive nilpotent subalgebra in $U_q(\g)$,
and the star designates the dual module of lowest weight. The subspace $V^+_Z$ is identified with the kernel of
the left ideal annihilating the lowest vector in $Z^*$.
We consider the pull-back of the contravariant form from $(V\tp Z)^+$  to $V^+_Z$ which we call extremal twist.
Regarded as a linear map from $V^+_Z$ to its dual vector space, this pull-back
relates two natural constructions of singular vectors in $V\tp Z$.

The extremal twist can be obtained as a representation of a universal element $\Theta_Z$ from a certain extension
of $U_q(\g)$, which itself  can be expressed through a lift of the inverted invariant pairing $Z\tp Z^*\to \C$.
Such a pairing is non-degenerate and unique up to a normalization, thanks to irreducibility of $Z$.
The element $\Theta_Z$  appeared before in the theory of dynamical twist, for $Z$  a parabolic Verma module relative to
a Levi subalgebra $\k\subset \g$, cf. \cite{EV,KM}.

When $\k$ is the Cartan subalgebra $\h\subset \g$ and $Z$ is an ordinary Verma module,
the inverse element $\Theta^{-1}_Z$ participated in construction of  dynamical Weyl group in \cite{EV}.
It equals  the shifted extremal projector $p_\g(\zt)$  of $U_q(\g)$, by the
highest weight $\zt$ of the module $Z$. In this paper we extend that relation to all irreducible $Z$
of highest weight,
provided certain regularity assumptions on the operator $p_\g(\zt)$ as
a trigonometric rational function of $\zt$ are fulfilled.
This finding reduces the problem of semi-simplicity of tensor products to
computing the determinant of $p_\g(\zt)$. The shifted extremal projector
is naturally interpreted as  the universal inverse of the contravariant form transferred from $(\>\cdot\> \tp Z)^+$
to $\Hom_{U_q(\g_+)}(Z^*, \>\cdot\>)$.

As an example, we consider a parabolic Verma module $Z$ relative to a Levi subalgebra $U_q(\k)\subset U_q(\g)$.
Such a module is parabolically induced from a finite dimensional $U_q(\k)$-module $X$ of highest weight $\zt$.
The factor $p_\k(\zt)$ entering  $p_\g(\zt)$  is invertible on the subspace of concern for every finite dimensional module $V$.
Then $p_\g(\zt)$ essentially reduces to a product $p_{\g/ \k}(\zt)$ of shifted $\s\l(2)$-projectors over the roots
from $\Rm^+_\g-\Rm^+_\k$. The universal extremal twist $\Theta_Z$ coincides with
$p_{\g/ \k}^{-1}(\zt)$ up to an invertible factor which degenerates to $1$ for scalar $X$.
The poles of $p_{\g/ \k}^{-1}(\zt)$ correspond to
reducible $Z$.

As another application, we compute the extremal twist for $Z$ the base module for a quantum sphere $\Sbb^{2n}$, \cite{M3},
and thereby prove that all tensor products $V\tp Z$ with finite dimensional quasi-classical $U_q(\s\o(2n+1))$-modules $V$
are completely reducible.

\section{Quantized universal enveloping algebras}
\label{SecPrelim}
Suppose that $\g$ is a complex semi-simple Lie algebra and  $\h\subset \g$ its Cartan subalgebra. Fix
a triangular decomposition  $\g=\g_-\op \h\op \g_+$  with  nilpotent Lie subalgebras
$\g_\pm$.
Denote by  $\Rm$ the root system of $\g$, and by $\Rm^+$ the subset of positive roots with basis $\Pi$
of simple roots.
Choose an inner product $(\>.\>,\>.\>)$ on $\h$ as a multiple of the restricted Killing form
and transfer it to $\h^*$ by duality.
For each $\la\in \h^*$  denote by  $h_\la$ an element of $\h$ such that $\mu(h_\la)=(\mu,\la)$, for all $\mu\in \h^*$.
Set $\la^\vee=2\frac{\la}{(\la,\la)}$ for non-zero $\la\in \h^*$.

By $U_q(\g)$ we understand the standard quantum group, cf. \cite{D,CP}. It is a $\C$-algebra
 with the set of generators $e_\al$, $f_\al$, and $q^{\pm h_\al}$, $\al \in \Pi$, obeying
$$
q^{h_\al}e_\bt=q^{(\al,\bt)}e_\bt q^{h_\al},
\quad
[e_\al,f_\bt]=\dt_{\al,\bt}\frac{q^{h_\al}-q^{-h_\al}}{q_\al-q_\al^{-1}},
\quad
q^{h_\al}f_\bt=q^{-(\al,\bt)}f_\bt q^{h_\al},\quad \al, \bt \in \Pi,
$$
where  $q_\al=q^{\frac{(\al,\al)}{2}}$ and $q^{h_\al}q^{-h_\al}=1=q^{-h_\al}q^{h_\al}$.
The elements  $e_\al$ and $e_{-\al}=f_\al$ satisfy the q-Serre relations
$$
\sum_{k=0}^{1-a_{\al \bt}}(-1)^k {1-a_{\al \bt}\choose k}_{q_\al}e_{\pm \al}^{k}e_{\pm \bt}e_{\pm \al}^{1-a_{\al\bt}-k}, \quad \al\not =\bt.
$$
We use the notation  $a_{\al \bt}=(\bt,\al^\vee)$ for the Cartan matrix,  and ${m \choose n}_q=\frac{[m]_q!}{[n]_q![m-n]_q!}$, where $[m]_q!=[1]_q\cdot \ldots \cdot [m]_q$. Here and throughout the paper we write  $[z]_q=\frac{q^z-q^{-z}}{q-q^{-1}}$ for $z\in \h+\C$.
The complex parameter $q\not =0$ is assumed not a root of unity.

Fix the comultiplication in $U_q(\g)$ as
\be
\Delta(f_\al)= f_\al\tp 1+q^{-h_\al}\tp f_\al,\quad \Delta(q^{ h_\al})=q^{h_\al}\tp q^{ h_\al},\quad\Delta(e_\al)= e_\al\tp q^{h_\al}+1\tp e_\al.
\label{coprod}
\ee
Then the antipode $\gm$ acts on the generators by the assignment $\gm( f_\al)=- q^{h_\al}f_\al$, $\gm( q^{h_\al})=q^{- h_\al}$, $\gm( e_\al)=- e_\al q^{-h_\al}$. The counit returns $\eps(e_\al)=\eps(f_\al)=0$, and $\eps(q^{ h_\al})=1$.

Denote by $U_q(\h)$,  $U_q(\g_+)$, $U_q(\g_-)$  the subalgebras in $U_q(\g)$  generated by, respectively, $\{q^{\pm h_\al}\}_{\al\in \Pi}$, $\{e_\al\}_{\al\in \Pi}$, and $\{f_\al\}_{\al\in \Pi}$. The algebra $U_q(\g)$ is a free $U_q(\g_-)-U_q(\g_+)$-bimodule
generated by $U_q(\h)$ and features a triangular decomposition  $U_q(\g)=U_q(\g_-)U_q(\h)U_q(\g_+)$
 as in the classical case.
The quantum Borel subgroups  $U_q(\b_\pm)=U_q(\g_\pm)U_q(\h)$ are Hopf subalgebras in $U_q(\g)$.

We will need the following involutive maps on $U_q(\g)$.
The assignment
\be
\si\colon e_\al\mapsto f_\al, \quad\si\colon f_\al\mapsto e_\al, \quad \si\colon q^{h_\al}\mapsto q^{-h_\al}
\label{sigma}
\ee
extends to an algebra automorphism of $U_q(\g)$ and coalgebra anti-automorphism.
The involution $\omega=\gamma^{-1}\circ \si$ preserves the comultiplication but flips the multiplication.

All $U_q(\g)$-modules are assumed left and diagonalizable over $U_q(\h)$. Given a module  $V$, we write
$V[\la]$ for its subspace of weight $\la \in \h^*$. This notation applies to any $U_q(\h)$-module as well.
We denote by $\La(V)\subset \h^*$ the set of weights of a $U_q(\h)$-module $V$.

\subsection{Contravariant form on $V\tp Z$ and extremal twist}
In this section we recall a criterion for a tensor product $V\tp Z$ to be  completely reducible, following \cite{M1}.
A symmetric bilinear form $\langle \>.\>,.\>\rangle$ on a module $Z$  is called contravariant with
respect to involution $\omega$ if
 $\langle x z, w\rangle=\bigl\langle z,\omega(x)w\bigr\rangle$ for all $z,w\in Z$ and all $x\in U_q(\g)$.
It is known that every highest weight module has a unique, up to a scalar multiplier, contravariant form,
which is non-degenerate if and only if the module is irreducible.
Let us recall its construction.
Let $\wp\colon U_q(\g)\to U_q(\h)$ denote the projection along $\g_-U_q(\g)+U_q(\g)\g_+$ facilitated by
the triangular decomposition.
If $Z$ is the Verma module with highest weight $\zt$ and the highest vector $1_Z$, then the form is defined
by $\langle x 1_Z, y 1_Z\rangle=\zt\left(\wp \bigl(\omega(x)y\bigr)\right)$ for all  $x,y\in U_q(\g)$.
Its kernel is the maximal proper submodule, therefore the form
transfers to any quotient module.

Suppose that  $X$ is a module of lowest weight $\xi$ and $Z$ is a module  of highest weight $\zt$.
We
extend the tensor product $X\tp Z$   to  $X\hat \tp Z$  as follows.
For $\bt \in \Z\Pi$, we define $(X\hat \tp Z)[\xi+\zt+\bt]$  as the vector space of formal sums
over $\mu,\nu\in \Z_+\Pi$ subject to  $\mu-\nu=\bt$ of tensors from $X[\mu+\xi]\tp Z[-\nu+\zt]$. Then $X\hat \tp Z$
consists of  finite linear combinations of elements from $(X\hat \tp Z)[\xi+\zt+\bt]$ with $\bt \in \Z\Pi$.
It is easy to see that the $U_q(\g)$-action on $X\tp Z$ extends to an action of
$X\hat \tp Z$. We also apply this construction to tensor products of diagonalizable
$U_q(\h)$-modules with finite dimensional weight spaces whose weights are bounded from below and, respectively, from above.

The contravariant form on $Z$ is equivalent to an invariant paring $Z\tp Z'\to \C$, where $Z'$ is
the opposite module of lowest weight $-\zt$. They are related by a linear isomorphism
$\id\tp \si_Z\colon Z\tp Z\to Z\tp Z'$, where $\si_Z(f 1_Z)=\si(f)1_{Z'}$ for  $f\in U_q(\g_-)$ and  $1_{Z'}$ is
 the lowest vector in $Z'$. If the form is non-degenerate, $Z'$ is isomorphic to
 $Z^*$, and there exists a $U_q(\g)$-invariant element (the innverse form) $\Sc\in Z'\hat\tp Z$.
The converse is also true.

\begin{propn}
  Suppose there exists a $U_q(\b_+)$-invariant element $\Sc\in Z'\hat \tp Z$
  such that $\Sc_1\langle 1_Z,\Sc_2\rangle =1_{Z'}$. Then $Z$ is irreducible,
   and $\Sc$ is the inverse invariant form.
  \label{lift is singular}
\end{propn}
\begin{proof}
  For any $f\in U_q(\g_-)$, one has $\omega(f)=(\gm^{-1}\circ \si )(f)\in U_q(\b_+)$. Then
$$
  \Sc_1\langle \Sc_2,f1_Z\rangle = \Sc_1\left\langle (\gm^{-1}\circ \si )(f)\Sc_2,1_Z\right\rangle
  = \si (f)\Sc_1\langle \Sc_2,1_Z\rangle=\si (f)1_{Z'}=\si_Z(f1_{Z}).
$$
In other words, the map $z\mapsto   \si_Z^{-1}(\Sc_1)\langle \Sc_2,z\rangle$
is identical on $Z$. Therefore the contravariant form is non-degenerate,
and $\left \langle z \tp w, (\si_Z^{-1}\tp \id)(\Sc)\right\rangle = \langle z, w\rangle $ for all $z,w\in Z$ as required.
\end{proof}

Suppose that $Z$ is irreducible and let $V$ be another irreducible module of highest weight.
Denote by $(V\tp Z)^+$ the span of singular vectors in $V\tp Z$, i.e. the space of $U_q(\g_+)$-invariants.
Define  canonical contravariant symmetric bilinear form on $V\tp Z$ as the product of contravariant forms on $V$ and $Z$.
\begin{thm}[\cite{M1}]
   The tensor product $V\tp Z$ is completely reducible if and only if the canonical form is non-degenerate when restricted
  to $(V\tp Z)^+$.
\label{canonical}
\end{thm}
\noindent
Note that the form is non-degenerate on the entire $V\tp Z$ but may not be so on $(V\tp Z)^+$.

To compute the restricted  form, we parameterize $(V\tp Z)^+$ with
a vector space $V^+_Z=\Hom_{U_q(\g_+)}(Z^*,V)$.
We identify  it with a subspace in $V$
annihilated by the left ideal $I^+_Z\subset U_q(\g_+)$ that kills the lowest vector in $Z^*$.
The annihilator $V^\perp_Z$ of $V^+_Z$ with respect to the contravariant form coincides with $\omega(I^+_Z)V$.
The linear map $\bar \dt\colon V\tp Z\to V$, $\bar \dt\colon v\tp z\mapsto v\langle 1_Z,z\rangle$
yields an isomorphism  $(V\tp Z)^+\to V^+_Z$. We denote by $\dt$ the inverse isomorphism.

Regard $Z$ as a module over $U_q(\g_-)$ and its right dual module ${}^*\!Z$ as one over $U_q(\g_+)$.
Denote by $\Fc\in U_q(\g_+)\hat \tp U_q(\g_-)$ a lift of  $\Sc\in {}^*\!Z\hat \tp Z$
under a linear section of the $U_q(\g_+)\tp U_q(\g_-)$-module
homomorphisms.
The element $\Theta_Z=\gamma^{-1}(\Fc_2)\Fc_1$ belongs to a certain extension of $U_q(\g)$
and gives rise to a linear map $\theta_{V,Z}\colon V^+_Z\to V/V^\perp_Z$ by
$$
\bigl \langle \theta_{V,Z}(v),w \bigr \rangle =\bigl \langle \Theta_Z(v),w \bigr\rangle,
$$
which is independent of the choice of lift $\Fc$ for $\Sc$.
\begin{propn}[\cite{M1}]
\label{V-Z-extr}
  The form $\bigl\langle \theta_{V,Z}(\>.\>),\>.\>\bigr\rangle$ is the pullback of the canonical form
  under the isomorphism $V^+_Z\to (V\tp Z)^+$.
\end{propn}

The vector space $V/V^\perp_Z$ can be identified with a subspace ${}^+\!V_Z\subset V$ that is transversal to $V^\perp_Z$
since the contravariant form  on $V$ is non-degenerate. If it is non-degenerate when restricted to $V^+_Z$, then $V=V^+_Z\op V^\perp_Z$, and one can set ${}^+\!V_Z=V^+_Z$.
Then the linear map $\theta_{V,Z}$ becomes an operator from $\End(V^+_Z)$.

\subsection{Braid group action on $U_q(\g)$ and a  Cartan-Weyl basis}
The algebra $U_q(\g)$ admits a Poincar\'{e}-Birkhoff-Witt (BPW)-like basis of ordered monomials in
"root vectors", which are constructed from the generators via an action
of the braid group, \cite{CP}, Ch.8.1. We need this  basis to
write extremal projectors of $U_q(\g)$ in next sections.

Define
$m_{\al \bt}=2,3,4,6$ for $\al,\bt \in \Pi$ if the entries of the Cartan matrix satisfy $a_{\al \bt}a_{\bt \al}=0,1,2,3$, respectively.
The braid group $\Bc_\g$ associated with $\g$ is generated by
elements $T_\al$, $\al \in \Pi$,  subject to the relations
$(T_\al T_\bt)^{m_{\al \bt}}=(T_\bt T_\al)^{m_{\al \bt}}$, $\al \not =\bt$.

The group $\Bc_q$ admits a homomorphism onto the Weyl group $\Wc$ with the kernel
generated by the relations $T_\al^2=1$, $\al \in \Pi$, and sending
$T_\al$ to the simple reflections $\si_\al \in \Wc$.
The length $\ell(T)$ of an element of $T\in \Bc_\g$ is defined as the minimal number of generators
in a presentation of $T$, which is called a reduced decomposition of $T$. The length
of an element of a Weyl group is defined similarly, as the number of simple reflections in a reduced
decomposition.
There is a length preserving section of the surjection $\Bc_\g\to \Wc$, which is a map of sets.

Define a $T_\al$-action  on generators of the quantum group by
$$
\quad T_\al(f_\al)=-q^{-h_\al}e_\al ,\quad T_\al(q^{h_\bt})=q^{h_\bt-a_{\al\bt}h_\al} ,  \quad T_\al(e_\al)=-f_\al q^{h_\al},
$$
$$
\quad T_\al(e_\bt)=\sum_{k=0}^{-a_{\al \bt}} \frac{(-1)^{a_{\al \bt}-k}q^{-k}}{[k]_q![-a_{\al \bt}-k]_q!}e_{ \al}^{-a_{\al\bt}-k}e_{\bt}e_{ \al}^{k}, \quad \al \not =\bt,
$$
$$
\quad T_\al(f_\bt)=\sum_{k=0}^{-a_{\al \bt}} \frac{(-1)^{a_{\al \bt}-k}q^{k}}{[k]_q![-a_{\al \bt}-k]_q!}f_{ \al}^{k}f_{\bt}f_{ \al}^{-a_{\al\bt}-k}, \quad \al \not =\bt.
$$
It  extends to an algebra automorphism of $U_q(\g)$. The operators $\{T_{\al}\}_{\al \in \Pi}$ amount to
an action of $\Bc_\g$ on $U_q(\g)$.

\begin{propn}[\cite{CP}, Prop. 8.1.6]
  Let $w\in \Wc$ be such that $\bt=w(\al)\in \Pi$ for some simple root $\al$.
  Then $T_\omega(e_{\al})=e_{\bt}$ and $T_\omega(f_{\al})=f_{\bt}$.
\label{braid_simple}
\end{propn}

Let $\si_{i_1}\ldots \si_{i_N}$,
where $\si_i=\si_{\al_i}$ and  $N=\#\Rm^+$, be
a reduced decomposition of the longest element of $\Wc$.
Define a sequence of positive roots by
$$
\mu^1=\al_{i_1},\quad \mu^2=\si_1(\al_{i_2}), \quad \ldots \quad \mu^N=\si_1\ldots \si_{N-1}(\al_{i_N}).
$$
This sequence induces a total ordering on $\Rm^+$, called normal,  such that any root of the form $\al+\bt$ with $\al,\bt\in \Rm^+$
is between $\al$ and $\bt$.
 Any subset $\tilde\Pi \subset \Pi$ generates a root subsystem  $\tilde \Rm\subset \Rm$.
There is a normal ordering where all roots form $\tilde \Rm^+$ are on the right of the roots from $\Rm^+\backslash \tilde \Rm^+$,
see e.g. \cite{Zh}, Exercise 1.7.10.
We will use this fact in Section \ref{Sec_Parabolic} in relation with Levi subalgebras.

A Cartan-Weyl basis in $U_q(\g)$ depends on a normal ordering and is defined as follows. The root $\mu^1$ is simple, so  $e_{\mu^1}$
and $f_{\mu^1}$ are the corresponding
Chevalley generators. For $k>1$ set
$$
e_{\mu^k}=T_{\al_{i_1}}\circ \ldots \circ T_{\al_{i_{k-1}}}(e_{\al_{i_k}}), \quad
f_{\mu^k}=T_{\al_{i_1}}\circ \ldots \circ T_{\al_{i_{k-1}}}(f_{\al_{i_k}}).
$$
Proposition \ref{braid_simple} guarantees that the simple root generators are included in this set.
It is known that normally ordered monomials in these elements  deliver a PBW basis in, respectively,  $U_q(\g_\pm)$ when $q$ is not a root
of unity.

Regarding $U_q(\g)$ as a $\C[q,q^{-1}]$-algebra, define an anti-automorphism $\tilde\omega$ by
\be
\tilde \omega\colon e_\al\mapsto f_\al, \quad\tilde \omega\colon f_\al\mapsto e_\al, \quad \tilde \omega\colon q^{h_\al}\mapsto q^{-h_\al},
\quad
\tilde \omega\colon q\mapsto q^{-1}.
\label{tilde-omega}
\ee
It commutes with the action of $\Bc_\g$.
\subsection{Properties of the Cartan-Weyl basis}
  For each $\mu\in \Rm^+$, one has
$
[e_\mu,f_\mu]=a_\mu\frac{q^{h_\mu}-q^{-h_\mu}}{q_\mu-q^{-1}_\mu}
$
for some $a_\mu\in \C^\times$. This relation facilitates an embedding $\iota_\mu\colon U_q\bigl(\s\l(2)\bigr)\to U_q(\g)$
determined by the assignment
$$
f\mapsto \frac{1}{a_\mu}f_\mu, \quad  e\mapsto e_\mu, \quad q^{h}\mapsto q^{h_\mu}, \quad
q\mapsto q_\mu,
$$
where $e,f,q^{h}$ are the standard generators of $U_q\bigl(\s\l(2)\bigr)$
satisfying  $q^h eq^{-h}=q^2 e $, $q^h fq^{-h}=q^{-2} f$, and $[e,f]=[h]_q$.
We denote by $U_q(\g^\mu)$ the image of $\iota_\mu$ under this embedding.

For $\al, \bt \in \Rm^+$ such that $\al<\bt$, denote by  $U^+_{\al,\bt}$ the $U_q(\h)$-submodule in $U_q(\g)$ under the
multiplication generated by the monomials $e_{\mu^i}^{k_i}\ldots e_{\mu^j}^{k_j}$
with $\al\leqslant \mu^i<\ldots<\mu^j\leqslant \bt $  and $\sum_{s}k_s>0$.
We set $U^-_{\al, \bt}=\tilde \omega (U^+_{\al, \bt})$ and denote $U^\pm_{\al\leqslant }=U^\pm_{\al,\mu^N}$ and $U^\pm_{\leqslant \al}=U^\pm_{\mu^1,\al}$.
We will also use the obvious notation $U^\pm_{< \al}$ and $U^\pm_{\al <}$
involving only the root vectors starting with the roots next to $\al$.

Given two vector subspaces in $A,B\subset U_q(\g)$ we denote $A\bullet B=A+B+AB$, where the last term is the linear span of products of
elements form $A$ and $B$.
\begin{propn}
The $U_q(\h)$-modules  $U^\pm_{\leqslant \al}$, $U^\pm_{\al \leqslant}$,
$U^-_{\leqslant \al}\bullet  U^+_{\bt \leqslant}$, and  $U^-_{\bt \leqslant}\bullet  U^+_{\leqslant \al}$ with  $\al<\bt$ are associative (non-unital) subalgebras in $U_q(\g)$
\label{PBW com_rel}
\end{propn}
\begin{proof}
Set  $\mu^i=\al$ and $\mu^j=\bt$ with  $i<j$ and put $\al'=\mu^{i+1}$ and $\bt'=\mu^{j-1}$, so that $\al<\al'\leqslant \bt'<\bt$.
Then the following relations hold, \cite{KT}:
$$
[e_\al,e_\bt]_{q^{(\al,\bt)}}\in U^+_{\al',\bt'},
\quad
[f_\bt,f_\al]_{\bar q^{(\al,\bt)}}\in  U^-_{\al',\bt'},
\quad
[e_\bt,f_\al]\in  U^-_{<,\al}\bullet  U^+_{\bt,<},
\quad
[e_\al,f_\bt]\in  U^-_{\bt, <}\bullet  U^+_{<,\al}.
$$
Here and further on we use the shortcut $\bar q=q^{-1}$.
Note that the second and fourth inclusions are obtained from the first and third by applying the automorphism $\tilde \omega$,
which flips $U_{\mu,\nu}^+$ and $U_{\mu,\nu}^-$.
Now the proof is straightforward.
\end{proof}
Note that these algebras have trivial subspace of zero weight.
\begin{propn}
\label{M-F}
  For all $\mu\in \Rm^+$,
$$
\Delta(e_\mu)\in e_\mu\tp q^{h_\mu} + 1\tp e_\mu + U^+_{\mu< }\tp U^+_{< \mu}, \quad
\Delta(f_\mu)\in f_\mu\tp 1 +  q^{-h_\mu}\tp f_\mu+U^-_{<\mu}\tp U^-_{\mu<}.
$$
\end{propn}
\begin{proof}
There is an invertible element $\tilde \Ru_\mu\in 1\tp 1+ U^-_{<\mu}\hat \tp U^+_{<\mu}$ such that
the coproduct $\Delta(e_\mu)$ can be expressed as
$$
 \Delta(e_\mu)=\tilde \Ru_{<\mu}(e_\mu \tp q^{h_\mu}+1\tp e_\mu)\tilde \Ru_{<\mu}^{-1}.
$$
This can be found in \cite{KT}, Proposition 8.3 (for a twisted coproduct as compared to ours, so our $\tilde \Ru$ is flipped).
As $\Delta(e_\mu)\in U_q(\g_+)\tp U_q(\b_+)$, it suffices to evaluate both  sides of this equality on the tensor product
of the right "universal Verma modules", i.e. the quotients of $U_q(\g)$ by the right ideal $J$ generated by $f_\al$, $\al \in \Pi$:
$$
\Delta(e_\mu)=(e_\mu \tp q^{h_\mu})\tilde \Ru_{<\mu}^{-1}+1\tp e_\mu \mod J\tp J.
$$
Pushing the left leg of $\tilde \Ru_{<\mu}^{-1}$ to the left with the use of the third inclusion from Proposition \ref{PBW com_rel} one proves the  left equality.
The other is obtained by applying  $\tilde \omega$,
which flips the comultiplication.
\end{proof}
\begin{propn}
  For any ordered sequence of positive roots $\mu^i<\ldots <\mu^k$
  the projection $\wp$ annihilates $U_{\mu^i\leqslant}^-\bullet U_{\leqslant\mu^i}^+\bullet \cdots \bullet U_{\mu^k\leqslant}^-\bullet U_{\leqslant \mu^k}^+$.
\label{absorbtion}
\end{propn}
\begin{proof}
The statement follows from the inclusion
\be
U_{i\leqslant}^-\bullet U_{\leqslant i}^+\bullet \cdots \bullet U_{k\leqslant}^-\bullet U_{\leqslant k}^+\subset
U_{i\leqslant}^-\bullet U_{\leqslant k}^+,
\label{-+}
\ee
where we write $U_{i\leqslant}^-=U_{\mu^i\leqslant}^-$ and $U_{\leqslant i}^+=U_{\leqslant \mu^i}^+$.
It is clearly true if $k=i$. Suppose it is proved for $k\geqslant i$.
Then
$$
U_{i\leqslant}^-\bullet U_{\leqslant i}^+\bullet \cdots \bullet U_{{k+1} \leqslant}^-\bullet U_{\leqslant {k+1}}^+\subset
U_{i\leqslant}^-\bullet U_{\leqslant k}^+\bullet U_{{k+1} \leqslant}^-\bullet U_{\leqslant {k+1}}^+\subset
U_{i\leqslant}^-\bullet U_{{k+1} \leqslant}^-\bullet U_{\leqslant  k}^+\bullet U_{\leqslant {k+1}}^+
$$
The right inclusion is due to Proposition  (\ref{PBW com_rel}). The result is contained in $U_{i \leqslant}^-\bullet U_{\leqslant {k+1}}^+$,
again by  Proposition  (\ref{PBW com_rel}).
Induction on $k$ completes the proof.
\end{proof}
\section{Extremal twist and extremal projector}
\label{Sec_con_form_proj}
We start with the case of $\g=\s\l(2)$ and normalize the inner product so that  $(\al,\al)=2$ for its only positive root $\al$.
Set $e=e_\al$, $f=f_\al$, and $q^{\pm h}=q^{\pm h_\al}$ to be the standard generators of $U_q(\g)$.
Extend $U_q(\g)$ to $\hat U_q(\g)$ by including infinite sums of elements from $\C[f]\C[e]$ of same weights
with coefficients in the field of fractions $\C(q^{\pm h})$.
Similar extension works for general semi-simple $\g$ making $\hat U_q(\g)$
an associative algebra, see e.g. \cite{KT}.

Define an element $p(t)\in \hat U_q\bigl(\s\l(2)\bigl)$ depending on a complex parameter $t$ by
\be
\label{translationed_proj}
p(t)=\sum_{k=0}^\infty  f^k e^k \frac{(-1)^{k}q^{k(t-1)}}{[k]_q!\prod_{i=1}^{k}[h+t+i]_q}
\in \hat U_q(\g).
\ee
It is stable under the involution $\omega$.

For every module $V$ with locally nilpotent action of $e$, the function $t\mapsto p(t)$ is a rational trigonometric
endomorphism of every weight space. On a module of highest weight
$\la$, 
it acts by
\be
p(t)v=c\prod_{k=1}^{l}\frac{[t-k]_q}{[t+\xi(h)+k]_q}v,
\label{proj_eigen}
\ee
where $v$ is a vector of weight $\xi=\la-l\al$ and $c=q^{-l \xi(h)+2l^2+l(l-1)}\not =0$.

For general $\g$ and $\mu\in \Rm^+$ let $p_\mu(t)$ denote the image of $p(t)$
in $\hat U_q(\g)$ under the embedding $\iota_\mu\colon \hat U_q(\s\l(2))\to \hat U_q(\g)$.
Put  $\la_i=2\frac{(\la,\mu^i)}{(\mu^i,\mu^i)}\in \C$ for $\la\in \h^*$ and $\mu^i\in \Rm^+$. Define
\be
p_\g(\la)=p_{\mu^1}(\rho_1+\la_1)\cdots p_{\mu^N}(\rho_N+\la_N),
\label{factorization}
\ee
assuming $\{\mu^i\}_{i=1}^N$  normally ordered.
It is independent of a normal ordering and turns to the extremal projector $p_\g$  at $\la=0$,  \cite{AST,KT}, which is the only element
of  $\hat U_q(\g)$ satisfying
$$
p_\g^2=p_\g, \quad e_\al p_\g =0 =p_\g f_\al , \quad \forall \al \in \Pi.
$$
Uniqueness implies that $p_\g$ is $\omega$-invariant.

Let $V$ and $W$ be vector spaces.
Suppose that $\C^k\ni \la\mapsto F(\la)\in \Hom(W,V)$ is a rational trigonometric function.
We say that $F(\la_0)$ admits a regularization if there is $\eta\in \C^k$ such that
the function $\C\ni t\mapsto F(\la_0+t\eta)$, is regular at $t=0$.
If its value is independent of $\eta$, then we say that  $F(\la_0)$ is well defined.

Furthermore, we say that $p_\g$ admits a regularization on a subspace $W$ of a $U_q(\g)$-module $V$
if so does $p_\g(\la)$ at $\la=0$ and the image of the regularized map $W\to V$ is in $U_q(\g_+)$-invariants.

Suppose that $V$ and $Z$ are irreducible modules of highest weights $\nu$ and, respectively, $\zt$. Let $1_V\in V$, $1_Z\in Z$
be their highest vectors.
\begin{propn}
   Suppose that $W\subset  V$ is a vector subspace such that $p_\g\colon W\tp 1_Z\to  (V\tp Z)^+$ admits a regularization.
   Then $p_\g(\zt)\colon W\to V^+_Z$ admits a regularization. Furthermore,
\label{twist-cocycle}
\be
\dt\circ p_\g(\zt)(w)=p_\g(w\tp 1_Z), \quad w\in W.
\label{tw-coc-form}
\ee
\end{propn}
\begin{proof}
By Proposition \ref{M-F}, for all $\al\in \Rm^+$ and  all $n\in \Z_+$, the coproducts satisfy
$$
\Delta (f_\al^n)= f_\al^n\tp 1\mod U_q(\g)\tp U^-_{\al\leqslant},
\quad
\Delta (e_\al^n)= e_\al^n\tp q^{n h_\al}\mod U_q(\g)\tp U^+_{\leqslant \al}.
$$
With  $\eta\in \h^*$, $t\in \C$, and $w\in W$,  evaluation of
$\bar \dt\bigr(p_\g(w\tp 1_Z)\bigr)=p_\g^{(1)}(t\eta)w \tp \bigl\langle 1_Z,p_\g^{(2)}(t\eta) 1_Z \bigr\rangle$ reduces to the
  replacement
$$
  \Delta(q^{h_\al})\to q^{h_\al}\tp q^{h_\al},\quad \Delta(f_\al) \to f_\al\tp 1, \quad \Delta(e_\al)\to e_\al  \tp q^{h_\al}
$$
in  $\Delta\bigl(p_\g(t\eta)\bigr)$ for each $\al \in \Rm^+$, because the remainder vanishes  in view of  Proposition \ref{absorbtion}. This calculation returns $p_\g(\zt+t\eta)(w)$,
which proves the assertion.
\end{proof}

From now  to the end of the section we assume that the map
\be
\pi\colon v\mapsto p_\g(v\tp 1_Z)\in (V\tp Z)^+
\label{inv_twist}
\ee
admits a regularization on ${}^+\!V_Z$.
Then we have a map
\be
\bar \theta_{V,Z}\colon {}^+\!V_Z\to V^+_Z, \quad  \bar \theta_{V,Z}=\bar \dt\circ \pi.
\label{inv_twist1}
\ee
  It is an immediate corollary of Proposition \ref{twist-cocycle} that
  $\bar \theta_{V,Z}$ defines a symmetric bilinear form $\langle \bar \theta_{V,Z}(\>.\>),\>.\>\rangle$ on ${}^+\!V_Z$,
  which is the pull-back of the canonical form on $(V\tp Z)^+$ under (\ref{inv_twist}).
We will prove that this form is essentially  the inverse to the form determined by $\theta_{V,Z}$.

\begin{thm}
\label{Shap-proj}
The bilinear  forms $\bigl\langle \theta_{V,Z}(\>.\>),\>.\> \bigr\rangle$ and $\bigl\langle\bar \theta_{V,Z}(\>.\>),\>.\> \bigr\rangle$
are non-degenerate simultaneously. In that case, they are inverse to each other.
\end{thm}

\begin{proof}
\vspace{10pt}
\noindent
Suppose that $\theta_{V,Z}$ is inverible and compute $\bigl\langle (\dt \circ \bar \theta_{V,Z}\circ \theta_{V,Z})(v), \dt(w)\bigr\rangle $
for a pair of vectors $v, w\in V^+_Z$
in two different ways as follows (we  put $u=\dt(w)$ below for short).\\
i) Applying the definition  (\ref{inv_twist1}) we find the matrix element  equal to
$\left\langle  p_\g\bigl(\theta_{V,Z}(v)\tp 1_Z\bigr),u\right\rangle $.
Presenting  $p_\g=1+\sum_{i} \phi_i\psi_i$, where $\phi_i\in U_q(\g_-)$ and $\psi_i\in U_q(\b_+)$
carry non-zero weight,
we continue with
$$
 \bigl\langle \theta_{V,Z}( v)\tp 1_\zt, u\bigr\rangle +\sum_{i}\bigl\langle \phi_i\psi_i(\theta_{V,Z}(v)\tp 1_\zt),u\bigr\rangle=
  \bigl\langle \theta_{V,Z}(v)\tp 1_\zt,u\bigr\rangle=\bigl\langle \theta_{V,Z}(v),w\bigr\rangle=\bigl\langle v,\theta_{V,Z}(w)\bigr\rangle.
$$
The sum on the left vanishes because
 $\bigl\langle \phi_i\psi_i(\ldots), u\bigr\rangle=\bigl\langle\psi_i(\ldots), (\eps\circ \omega )(\phi_i) u\bigr\rangle =0$.\\
ii) By Proposition \ref{V-Z-extr} the matrix element in question is equal to
 $\langle \bar \theta_{V,Z}\circ \theta_{V,Z}(v),\theta_{V,Z}(w)\rangle $.
Since $\theta_{V,Z}$ is invertible, the image of $\theta_{V,Z}$ is ${}^+\!V_Z$,
and $\bar \theta_{V,Z}\circ \theta_{V,Z}=\id$ on $V^+_Z$.

\vspace{10pt}

Now suppose that $\bar \theta_{V,Z}$  is invertible
and evaluate  $\bigl\langle\dt\circ \bar \theta_{V,Z}(v), \dt\circ\bar \theta_{V,Z}(w)\bigr\rangle$
for $v,w\in {}^+\!V_Z$  in two different ways as follows.
\\
i)
On the one hand, it is equal to
$$
\bigl\langle p_\g(v\tp 1_Z), p_\g(w\tp 1_Z)\bigr \rangle=\bigl\langle v\tp 1_Z, \omega(p_\g)\circ p_\g(w\tp 1_Z)\bigr \rangle
=\Bigl\langle v,\bar \dt\bigl( p_\g(w\tp 1_Z)\bigr) \Bigr\rangle = \bigl\langle v, \bar \theta_{V,Z}(w) \bigr\rangle.
$$
ii) On the other hand, it is equal to  $\bigl\langle \theta_{V,Z}\circ \bar \theta_{V,Z}(v),\bar \theta_{V,Z}(w) \bigr\rangle$
 by Proposition \ref{V-Z-extr}.
Since the image of $\bar \theta_{V,Z}$ is $V^+_Z$, one arrives at   $\theta_{V,Z}\circ \bar \theta_{V,Z}=\id$ on ${}^+\!V_Z$.
\end{proof}
It follows that the regularization $\lim_{t\to 0}p_\g(t\eta)|_{V^+_Z\tp 1_Z}$ may depend on $\eta$ only if the contravariant form is degenerate on $(V\tp Z)^+$.
\subsection{On regularization of extremal projector}
Regularization of the extremal projector is crucial for application of Theorem  \ref{Shap-proj} to calculation of the extremal twist.
In this section we point out some facts of practical use.

It is natural to employ decomposition of $p_\g(\la)$ to a product of the root factors  (\ref{factorization}).
\begin{propn}
\label{reg_proj}
Let $V$ be a $U_q(\g)$-module and put $W=V[\mu]$  for some weight $\mu\in \La(V)$.
Fix a normal order on $\Rm^+$ and suppose that  $p_\al(\rho_\al)$ are well defined
on $W$ for all $\al \in \Rm^+$. Then the operator $p_\g(0)=\prod_{\al \in \Rm^+}^< p_\al(\rho_\al)$ is well defined on $W$ and independent of the normal order.
\end{propn}
\begin{proof}
Each factor in $p_\g(\la)$ corresponding to a root $\al \in \Rm^+$ depends on $\la$ through a regular function $\la\mapsto (\la,\al^\vee)$
and  is well defined once admits a regularization.
This implies the assertion.
\end{proof}
Note that one has to consider the entire weight space for $W$ because it is {\em a priori} invariant under all root factors
in $p_\g(\la)$.
Clearly the statement holds true for $W$ a sum of weight spaces.

\begin{propn}
  For any $r\in \C$ the operator $p_\al(r)$, $\al\in \Rm^+$, is well defined on a subspace of weight $\xi$ satisfying
  $(\la,\al^\vee)+r\in \Z_+$,
   in any $U_q(\g_+^\al)$-locally finite $U_q(\g^\al)$-module.
\label{easy_case}
\end{propn}
\begin{proof}
  Denominators $\prod_{i=1}^{k}[(\la,\al^\vee)+r+i]_{q_\al}$ in (\ref{translationed_proj}) do not vanish on such weight spaces
  ($q$ is not a root of unity).
\end{proof}
Although  Proposition \ref{easy_case} is rather crude it proves to be useful.
We also need more delicate criteria, rather in a more special situation.
\begin{lemma}
 \label{reg_fin_dim}
  Let $V$ be a finite dimensional $U_q(\g^{\al})$-module, $\al \in \Rm^+$.
  For any $r\in \N$ the operator $p_\al(r)$ is well defined on $V$.
\end{lemma}
\begin{proof}
We can assume that $V$ is irreducible.
Let $\mu=\frac{m}{2}\al$, $m \in \Z_+$,  be the highest weight of $V$, and $t\in \C$.
The eigenvalue of $p_\al(t+r)$ on the subspace of weight $\mu-l \al$ with $0\leqslant l\leqslant m$ is proportional to $\frac{\prod_{k=1}^{l}[t+r-k]_{q_\al}}{\prod_{k=0}^{l-1}[t+r+m-l-k]_{q_\al}}
$, cf. (\ref{proj_eigen}).
As $t\to 0$, the denominator may have the only vanishing factor that  corresponds to non-negative  $k=m+r -l \leqslant l-1$.
But it
is cancelled  by a factor in the enumerator unless $r>l$ which contradicts
the  previous inequality in view of  $l\leqslant m$.
\end{proof}
As a consequence, we obtain the following important special case.
\begin{propn}
  The extremal projector $p_\g$ is well defined  on every dominant weight space of a locally finite $U_q(\g)$-module $ V$.
\label{reg_loc_fin}
\end{propn}
\begin{proof}
We can assume that $V$ is irreducible. Consider the case  $\g=\g^\al\simeq \s\l(2)$ first.
It follows from  (\ref{proj_eigen}) that  $p_\g(\la)$ is regular on $V$  at $\la=0$.
For all $\xi \in \La(V)$ with $(\xi,\al^\vee)\geqslant 0$ the operator $p_\g(0)$ projects $V[\xi]$
to the space of highest weight.
For general $\g$,  all root factors in $p_\g$ are well defined by Lemma \ref{reg_fin_dim} and independent of the normal order
by Proposition \ref{reg_proj}. For each simple $\al$ choose a normal order such that $\al$ is in the left-most position.
Then $p_\g(0)$ has the factor $p_\al(1)$ on the left that maps all $V[\xi]$ with $(\xi,\al^\vee)\geqslant 0$
to $\ker e_\al$.
Therefore the operator $p_\g(0)$ restricted to dominant weight spaces of $V$ takes values in $U_q(\g_+)$-invariants.
\end{proof}
Remark that although $p_\g(0)$ is well defined on every finite dimensional module by Proposition \ref{reg_proj},
non-dominant weight spaces are not generally killed by $p_\g(0)$, so it is not a projector to $U_q(\g_+)$-invariants.
That can be seen already on the example of $\g=\s\l(2)$, by examining (\ref{proj_eigen}) for $\dim V>2$ and  $\xi(h)\leqslant -2$.

\section{The case of  parabolic Verma modules}
\label{Sec_Parabolic}
In this section we compute extremal twist for tensor product of finite dimensional and parabolic Verma modules.
The key issue is regularization of the operator $\Delta(p_\g)$ restricted to a certain subspace in
the tensor product.
We address it relaxing the assumption that parabolic modules involved are irreducible.

\subsection{Regularization of extremal projector}
\label{RegExtPr_parabolic}
Let $\k\subset \g$ be a Levi subalgebra, i.e. a reductive Lie algebra of maximal rank
whose basis of simple roots $\Pi_\k$ is a subset in $\Pi_\g=\Pi$.
A weight $\la$ defines a one-dimensional  representation $\C_\la$ of $U_q(\k)$ if and only if $q^{2\la(h_\al)}=\pm 1$ for
all $\al\in \Pi_\k$.
We can assume plus as the other cases can be covered by tensoring with appropriate one-dimensional
$U_q(\g)$-modules.

Let  $\c\subset \h$ denote
the center of $\k$ and $\c^*\subset \h^*$ the subset of weights $\la$ such that $\la(h_\al)=0$ for all $\al\in \Pi_\k$.
Identify the Cartan subalgebra of the semi-simple part of $\k$ with the orthogonal complement of
$\c$ in $\h$. Then the weight lattice $\Lambda_\k$ of $\k$ is a subset of $\nu\in \h^*$ such that  $(\nu, \al^\vee)\in \Z$, $\al\in \Pi_\k$ and $(\nu,\la)=0$,
 $\la\in \c^*$.

Fix $\xi\in \La^+_\k$, $\la \in \c^*$ and set $\zt=\xi+\la$. Denote by $X$ the finite dimensional irreducible $U_q(\k)$-module of highest weight $\xi$.
Fix $Z$ to be  the quotient of the Verma module $\hat M_{\zt}$ with the highest vector $1_\zt$ by the sum of the submodules
$U_q(\g)e_\al^{m_\al+1}1_{\zt}$, where $m_\al=(\xi,\al^\vee)\in \Z_+$ for $\al\in \Pi_\k$.
The module $Z$ is locally finite over $U_q(\k)$, \cite{M2}, and $U_q(\k)1_Z\simeq X\tp \C_\la$,
where $1_Z$ is highest  vector of $Z$.

The annihilator $I^-_{Z}\subset U_q(\g_-)$ of the vector $1_Z\in Z$ is independent of $\la$ as well as
the left ideal $I^+_Z=\si(I^-_Z)\subset U_q(\g_+)$.
In this Section, we set $V^+_Z$ to be the kernel of  $I^+_Z$ in $V$.
When $Z$ is irreducible, $V^+_Z$ is the generalized extremal subspace parameterizing singular vectors in $V\tp Z$, as before.
We have $V^+_Z=V^+_X$, where $V$ in the right-hand side is considered as a $U_q(\k)$-module,
so $V^+_X$ is parameterizing singular vectors in $V\tp X$. By ${}^+\!V_Z$ we understand a subspace in $V$
that is dual to $V_Z^+$ with respect to the contravariant form. This notation is also compatible with restriction of
the representation to $U_q(\k)$, that is, ${}^+\!V_Z={}^+\!V_X$.

 Denote by $\c^*_{reg}$ the set of weights $\la\in \c^*$
such that $(\la,\al^\vee)\not \in \Z$ for all $\al \in \Rm^+_{\g/\k} =\Rm^+_{\g} - \Rm^+_{\k} $. It is an open subset in $\c^*$
in the Euclidean topology.
Choose a normal ordering on $\Rm_\g^+$ such $\Rm_{\g/\k}^+<\Rm_\k^+$. Denote by $p_\k(\zt)$ the shifted extremal projector of $U_q(\k)$
  and by $p_{\g/ \k}(\zt)$ the product
$$
p_{\g/ \k}(\zt)=\prod_{\mu^i\in \Rm^+_{\g/\k}}^{<}p_{\mu^i}(\rho_i+\zt_i).
$$
Note that $p_\k(\zt)=p_\k(\xi)$ is independent of the summand $\la\in \c^*$.
The factorization
$
p_\g(\zt)=p_{\g/ \k}(\zt)p_{\k}(\xi)
$
facilitates regularization of $p_\g(\zt)$ on
${}^+\!V_Z$ and $(V\tp Z)^+$ as explained next.
\begin{propn}
  For any finite dimensional $U_q(\k)$-module $Y$, the operator $p_\k(\xi)$ is well defined and invertible on ${}^+\!Y_X$. Furthermore,
  $
    \theta_{Y,X}=p_\k^{-1}(\xi)
  $.
\end{propn}
\begin{proof}
All weights of ${}^+\!Y_X\tp 1_X$ are dominant when restricted to $\Rm_{\k}$.
  Applying Proposition \ref{reg_loc_fin} to (the semi-simple part) of $\k$, we conclude that $p_\k$ is well defined on ${}^+\!Y_X\tp 1_X\subset Y\tp X$
  and therefore $p_\k(\xi)$ is well defined on ${}^+\!Y_X$, by Theorem \ref{Shap-proj}.
  It is invertible and its inverse equals $\theta_{Y,X}$, since $Y\tp X$ is completely reducible.
\end{proof}
\begin{corollary}
   For each $\la \in \c^*_{reg}$ the linear maps
  $p_\g(\zt)\colon {}^+\!V_Z\to V$, $v\mapsto p_\g(\zt)v$ and $\pi\colon {}^+\!V_Z\to (V\tp Z)^+$,
   $v\mapsto p_\g(v\tp 1_Z)$ are well defined. Furthermore,
   $$
   p_\g(\zt)v=p_{\g/\k}(\zt)\bar \theta_{V,{X}}v, \quad p_\g(v\tp 1_Z)= p_{\g/\k}(0)p_\k(v\tp 1_Z)
   $$
   for all $v\in {}^+\!V_Z$.
\label{regularization_parabolic}
\end{corollary}
\begin{proof}
  The subspace ${}^+\!V_Z\tp 1_Z$ lies in the finite dimensional $U_q(\k)$-submodule isomorphic to $V\tp X$, so $p_\k$ is well defined on it.
  Moreover, $p_{\k/\g}(0)$ is well defined on ${}^+\!V_Z\tp 1_Z$ as none of denominators in (\ref{translationed_proj}) turns zero.
  Now the proof follows from Proposition \ref{twist-cocycle}.
\end{proof}
Remark that in the special case of $\xi=0$ corresponding to a scalar parabolic module $Z$
one can take ${}^+\!V_Z = V^+_Z$. Then  $p_\k(\xi)$ is identical on ${}^+\!V_Z$
and drops from the factorization.
\subsection{Extremal twist for parabolic modules}
\label{Sec_Semi-simplicity}
In order calculate the extremal twist, we first
work out a necessary condition for  parabolic Verma modules to be irreducible that is fulfilled for generic highest weight.
Note that complete irreducibility criteria for classical parabolic Verma modules are given  in \cite{Jan2}.
We do not appeal to deformation arguments but make  use of the relation (\ref{tw-coc-form}) between the inverse invariant pairing
and extremal projector.

The idea of  our approach originates from Proposition \ref{lift is singular}.
However, we cannot directly apply the extremal projector to construct a singular vector
in $Z'\tp Z$ ($Z'$ is the opposite parabolic module of lowest weight $-\zt$) since weights of $Z'$ are not bounded from above. We
 approximate $Z'$ with a sequence of finite dimensional modules $\{V_\mu\}$ and modify
Proposition \ref{lift is singular} accordingly.
We then construct $\Sc\in Z'\hat \tp Z$ as a projective limit of singular vectors in $V_\mu\tp Z$.

Suppose that $u=u^1\tp u^2\in V\tp Z$ (Sweedler notation) is a singular vector such that $\bar \dt(u)=v\in V$ is not zero.
Define a linear map $\psi_{v}\colon Z\to V$ as $\psi_{v}(z)=u^{1}\langle u^2,z\rangle$, for all $z\in Z$.
It factors to a composition
$Z\to Z^*\to V$, where the first arrow is the contravariant form
regarded as a linear map from $Z$ to its restricted ($U_q(\h)$-locally finite) dual $Z^*$.
\begin{propn}
\label{part_iso}
  For any element $f\in U_q(\g_-)$ of weight $-\bt$, $\psi_{v}(f1_Z)=q^{-(\zt+\mu,\bt)}\si(f)v$,
  where $\mu$ is the weight of $v$. In particular, $v\in V^+_Z$.
\end{propn}
\begin{proof}
  It is sufficient to prove the equality for $f$ a monomial in Chevalley generators.
For simple $\bt \in \Pi$ one has
$$
\bigl(1\tp \omega(f_\bt )\bigr)u=-(1\tp q^{-h_\bt}e_\bt )u=-\bigl(\gm(e_\bt)\tp q^{-h_\bt} \bigr)u=(e_\bt q^{-h_\bt}\tp q^{-h_\bt} )u
=\bigl(\si(f_\bt) \tp 1\bigr)q^{-h_\bt}u.
$$
This implies $\bigl(1\tp \omega(f)\bigr)u=q^{-(\zt+\mu,\bt)}\bigl(\si(f)\tp 1\bigr)u$ for all $\bt$ and all monomial $f$.
Now the proof is immediate.
\end{proof}

Now regard $\La^+_\k$ as a natural sublattice in $\La^+_\g$ and for fixed $\xi \in \La^+_\k$  define $\c^*_{\xi,\Z}$ as the set of
integral weights
$\xi+\la$ with  $\la\in \La_\g\cap \c^*$. In other words, $\c^*_{\xi,\Z}$ is the shift by $\xi$ of the sublattice generated by
the fundamental weights dual to $\Pi_\g^\vee-\Pi_\k^\vee$.
Introduce a partial ordering on  $\c^*_{\xi,\Z}$ by setting $\nu\prec \nu$ if $(\nu,\al^\vee)< (\mu,\al^\vee)$ for all $\al\in \Pi_{\g}-\Pi_{\k}$.
Let $\c^*_{\xi,\Z_+} \subset \c^*_{\xi,\Z}$ be the subset of dominant weights.
For $\mu\in \c^*_{\xi,\Z_+}$ set $V_\mu$ to be the finite dimensional $U_q(\g)$-module of lowest weight $-\mu$.
Its highest weight is $-w(\mu)$, where  $w$ is the longest element of the Weyl group of $\Rm_\g$.

\begin{propn}
  For all  $\mu,\nu\in  \c^*_{\xi,\Z_+}$,  there is an inclusion $Z^+_{V_\mu}\subset Z^+_{V_\nu}$ once $\mu\prec \nu$.
\end{propn}
\begin{proof}
 The left ideal $I^+_{V_\mu}$ is generated by $\{e_\al^{m_\al'+1}\}_{\al\in \Pi}$
  with $m_\al'=m_{-w(\al)}$, where $m_\bt=(\mu,\bt^\vee)\in \Z_+$, cf. \cite{Jan1}.
  Clearly $I^+_{V_\mu}\supset I^+_{V_\nu}$ if $\mu\prec \nu$. Then
  $Z^+_{V_\mu}=\ker(I^+_{V_\mu})\subset \ker(I^+_{V_\nu})=Z^+_{V_\nu}$ as required.
\end{proof}
Denote by  $J_\mu^+\supset I^+_Z$ the left ideal in $U_q(\g_+)$ annihilating the lowest vector in $V_\mu$. It is generated by $\{e_\al^{m_\al+1}\}_{\al\in \Pi}$
  with $m_\al=(\mu,\al^\vee)\in \Z_+$.
There is  a $U_q(\g_+)$-invariant projection
$\wp_{\mu}\colon Z'\to V_{\mu}\simeq U_q(\g_+)/J_\mu^+$.
The following lemma facilitates approximation of $Z'$ with an increasing sequence of $V_\mu$.
\begin{lemma}
\label{weight_spaces}
  For each $\bt\in \Z_+\Pi_{\g}$ there exists $\mu \in \c^*_{\xi,\Z_+}$ such that
$\dim V_\mu[-\mu+\bt] = \dim Z'[-\zt+\bt]$.
\end{lemma}
\begin{proof}
 It is sufficient to take $\mu$ with $m_\al$ higher than the height of $\bt$ for all $\al\in \Pi_\g-\Pi_\k$.
 Then the kernel of $\wp_{\mu}$ has no weight $\bt$.
\end{proof}
\noindent
Since $J_{\nu}^+\subset J_\mu^+$ for $\mu\prec \nu$,
the projection $\wp_\nu$ factorizes to $\wp_\mu=\wp_\nu\circ \wp_{\nu,\mu}$ with
an $U_q(\g_+)$-equivariant projection $\wp_{\nu,\mu} \colon V_{\nu}\to V_{\mu}$.
Lemma \ref{weight_spaces} then implies $\cap_{\mu} \J_\mu^+=I^+_Z$, where the intersection
is over $\mu\in \c^*_{\xi,\Z_+}$, and $ Z'$ is a projective limit of $U_q(\g_+)$-modules $V_\mu$.

The lowest vector $v_\mu\in V_\mu$ belongs to $(V_\mu)^+_Z$, and Corollary \ref{regularization_parabolic} implies
that a singular vector $u_\mu= p_{\g}(p_{\g}^{-1}(\zt)v_\mu\tp 1_\zt)$ with $\bar \dt (u_\mu)=v_\mu$ is well defined
for $\la \in \c^*_{reg}$ (it follows from (\ref{proj_eigen}) that $p_{\g/\k}(\zt)$ and therefore $p_{\g}(\zt)$
are invertible for such $\la$).

\begin{corollary}
\label{part_irred}
  The module $Z$ is irreducible once $\la \in \c^*_{reg}$.
\end{corollary}
\begin{proof}
Let  $\bt \in \Z_+\Pi_{\g}$ be such that $Z[\zt-\bt]\not =\{0\}$. Take $\mu\in \La^+_\xi$ sufficiently large so that $V_\mu[-\mu+\bt]\simeq Z'[-\zt+\bt]$ (that is possible in view of  Lemma \ref{weight_spaces}).
 The map $\psi_{v_\mu}$ is an isomorphism between $Z[\zt-\bt]$ and $V_\mu[-\mu+\bt]$ by Proposition \ref{part_iso}.
 Therefore the contravariant form is non-degenerate on $Z[\zt-\bt]$ and hence on all weight subspaces of $Z$.
\end{proof}
\begin{propn}
\label{twist_parab}
  Let $Z$ be an irreducible parabolic Verma module of highest weight $\zt=\xi+\la\in \La^+_\k\op \c^*$. For every finite dimensional module $V$,
  the extremal twist  $\theta_{V,Z}$ is the operator $p_{\k}(\xi)^{-1}p_{\g/\k}(\zt)^{-1}$ restricted to $V^+_Z$.
\end{propn}
\begin{proof}
This is true for $\la\in \c^*_{reg}$ by Corollary \ref{part_irred}.
  The operator $\theta_{V,M_\la}$ is a rational trigonometric function of $\la\in \c^*$ coinciding with $p_{\k}(\xi)^{-1}p_{\g/\k}(\zt)^{-1}$
  on an open subset $\c^*_{reg}\subset \c^*$ and therefore on $\c^*$.
\end{proof}
\noindent
As a consequence we conclude that  if $\la\in  \c^*$ is a pole of the map
$p_{\g/\k}(\xi+\la)^{-1}\colon V^+_Z\to V$ then the module $Z$
  is reducible.

\subsection{Equivariant star product on Levi conjugacy classes}
In this section we give an expression for an equivariant star product on homogeneous spaces with Levi stabilizer subgroup.
Such a space is realized as a conjugacy class, and the Poisson structure is restricted from the  Semenov-Tian-Shansky bracket
on the total group. The corresponding star product was constructed in \cite{EEM} with the help of dynamical twist, which
reduces to the inverse contravariant
form, \cite{AL}. While that solves the problem in principle, an explicit expression of the inverse form for a general parabolic
module is unknown. In this section we give an alternative  formula for the star product in terms of extremal projector, which
is absolutely explicit. The idea of our approach is close to \cite{Khor} (for the special case of $\k=\h$) and
based on  relation (\ref{tw-coc-form}).

In this section we assume that $\xi=0$ and $\zt=\la$. For the module $Z$ we take the scalar parabolic Verma module $M_\la$ of highest weight $\la$.
Then $V^+_{M_\la}$ is the space $V^{\k_+}$ of $U_q(\k_+)$-invariants in $V$.
One can check that the contravariant form is non-degenerate when restricted to $V^{\k_+}$ so we can choose
${}^+\!V_{M_\la}=V^{\k_+}$.  The projector $p_\k$ reduces to identity on $V^{\k_+}$, which gives
 $p_\g(\la)=p_{\g/\k}(\la)\in \End(V^{\k_+})$.

Let $\Ac$ denote the quantized Hopf  algebra of polynomial functions on an algebraic group $G$ with the Lie algebra $\g$.
The quantum group $U_q(\g)$ acts on $\Ac$ by right translations, according to   $x\tr a=a^{(1)}(x,a^{(2)})$.
Fix a weight  $\la \in \c^*_{reg}$, so that $M_\la$ is irreducible, and let $\Fc\in U_q(\g_+)\hat\tp U_q(\g_-)$ be a lift
of the inverse invariant paring $M_\la\tp M_\la'\to \C$.
It defines a
bi-differential operator on $\Ac$ by
\be
\label{star_prod0}
\Ac\tp \Ac \stackrel{\Fc }{\longrightarrow} \Ac\tp \Ac \stackrel{\cdot }{\longrightarrow} \Ac,
\ee
where $\cdot$ is the multiplication on $\Ac$. This operation is parameterized by $\la$ and it is known to be associative
when restricted to the subspace $\Ac^\k$ of $U_q(\k)$-invariants in $\Ac$.

Denote by $\Phi$ the composition map
$$
\Ac\tp \Ac\tp M_\la\stackrel{\langle 1_\la, \>. \>\rangle}{\longrightarrow} \Ac\tp \Ac \stackrel{\cdot }{\longrightarrow} \Ac,
$$
where the left arrow is the contravariant pairing of the $M_\la$-factor with the highest vector $1_\la$.
The formula (\ref{tw-coc-form}) in combination with regularization  of Section \ref{RegExtPr_parabolic}
gives the following presentation of the star product in terms of the Zhelobenko cocycle.
\begin{propn}
 The star-product on $\Ac^\k$ restricted from (\ref{star_prod0}) is presentable as
\be
f\star g= \Phi \Biggl( p_{\g/\k}(0)\Bigl(p_{\g/\k}^{-1}(\la)f\tp p_{\g/\k}(0)
\bigl(p_{\g/\k}^{-1}(\la) g\tp 1_\la \bigr)\Bigr)\Biggr),\quad f,g\in \Ac^\k,
\label{star_prod}
\ee
where the action of $U_q(\g)$ on $\Ac$ is $\tr$.
\end{propn}
\begin{proof}
Given a finite dimensional module $V$ and $v\in V^{\k}\subset V^{\k_+}$, one has
$$\Fc(v\tp 1_\la) = p_{\g/\k}(0)\bigl(p_{\g/\k}^{-1}(\la)v\tp 1_\la\bigr).$$
The vector in the right-hand side is $U_q(\k)$-invariant and generates a submodule isomorphic to $M_\la\subset V\tp M_\la$, so
one can iterate this operation with $w\in W^\k$ for another finite dimensional module $W$ and get a vector in $W\tp V\tp M_\la$.
Pairing of the $M_\la$-factor
with $1_\la$ is $U_q(\k)$-invariant and yields a tensor $\Fc (w\tp v)\in (W\tp V)^\k$.

Now take $f$ and $g$ from $\Ac^\k$, which is a direct sum of finite dimensional modules thanks to the
Peter-Weyl decomposition.  Then
$$
 \Fc\bigl(f\tp \Fc(g\tp  1_\la )\bigr)
=p_{\g/\k}(0)\Bigl(p_{\g/\k}^{-1}(\la)f\tp p_{\g/\k}(0)\bigl(p_{\g/\k}^{-1}(\la) g\tp 1_\la \bigr)\Bigl).
$$
Applying $\Phi$ yields  $f\star g$ in the left-hand side.
\end{proof}

\section{Application to vector bundles on quantum spheres}
We conclude this presentation by illustrating Theorems \ref{canonical} and \ref{Shap-proj}
with an example relevant to quantum even sphere, \cite{M2}.
Here $Z$ is fixed to a base module that supports the quantization
of $\C[\Sbb^{2n}]$ as a subalgebra in $\End_\C(Z)$, \cite{M3}. The module $V$ varies over all equivalence classes of finite dimensional quasi-classical irreducible representations of $U_q\bigl(\s\o(2n+1)\bigr)$.
Unlike in Section \ref{Sec_Parabolic}, the subspaces $V^+_Z\subset V$ are hard to evaluate while their reciprocals
$Z^+_V\subset Z$ are known from \cite{M2}, which enables us to  compute $\theta_{Z,V}$ via (\ref{proj_eigen}) and (\ref{factorization}).
Thus Theorem \ref{canonical} benefits from alternative parameterizations of singular vectors
that prove to be most convenient for
particular calculations.

In this section, we fix $\g=\s\o(2n+1)$ and $\k=\s\o(2n)\subset \g$. Note that
there is no natural quantization  of  $U(\k)$ as a subalgebra in $U_q(\g)$, contrary to the case of Levi $\k$.
Let $\{\ve_i\}_{i=1}^n$ denote the orthonormal basis of short roots in $\Rm^+$. We
enumerate the basis of simple positive roots as $\al_n=\ve_n-\ve_{n-1},\ldots, \al_2= \ve_2-\ve_{1}, \al_1=\ve_1$.
We choose $\la\in \h^*$ such that $q^{2(\la,\ve_i)}=-q^{-1}$ for all   $i=1,\ldots, n$
and define $Z$ as the module of  highest weight $\la$ whose canonical generator $1_Z$ is annihilated by $f_{\al_i}$ with
$i>1$ and by $[[f_{\al_2},f_{\al_1}]_q,f_{\al_1}]_{\bar q}$, $\bar q=q^{-1}$.
Set
$e_{\ve_{1}}=e_{\al_{1}}$ and  $f_{\ve_{1}}=f_{\al_{1}}$ and
furthermore
$$
e_{\ve_{i+1}}=[e_{\al_{i+1}},e_{\ve_i}]_{q}, \quad
f_{\ve_{i+1}}=[f_{\ve_i}, f_{\al_{i+1}}]_{\bar q}
$$
for $i>1$.
 Weight vectors $ f_{\ve_1}^{m_1}\ldots f_{\ve_{n}}^{m_{n}}1_Z$ with $m_i$ taking all
possible values in $\Z_+$ deliver an orthogonal basis in  $Z$, \cite{M3}.

The module $Z$ is a quotient of a parabolic Verma module relative to the Levi subalgebra $\l\subset \g$
with the basis of simple roots  $\Pi_\l=\{\al_i\}_{i=2}^n$. Therefore it is locally finite over $U_q(\l)$.

Fix a finite dimensional $U_q(\g)$-module $V$  of highest weight $\nu$ and put $\ell_i=(\nu,\al_i^\vee)\in \Z_+$, $i=1,\ldots, n$.
The ideal $I^+_V$ determining $Z^+_V=\ker I^+_V\subset Z$, is generated by $\{e_{\al_i}^{\ell_i+1}\}_{i=1}^n$.
There is an orthogonal decomposition $Z=Z^+_V\op \omega(I^+_V)Z$
with
$$
Z^+_V =\Span\{ f_{\ve_1}^{m_1}\ldots f_{\ve_{n}}^{m_{n}}1_Z\}_{m_1\leqslant \ell_1, \ldots, m_n\leqslant \ell_n },\quad
\omega(I^+_V)Z =\Span\{ f_{\ve_1}^{k_1}\ldots f_{\ve_{n}}^{k_{n}}1_Z\}_{k_1, \ldots, k_n},
$$
where  $k_i>\ell_i$ for some $i=1,\ldots,n$.

The  weight $\nu$ is expanded in the orthogonal basis $\{\ve_i\}_{i=1}^n$ as
$$
\nu=\frac{\ell_1}{2}\sum_{i=1}^{n}\ve_i+\sum_{i=2}^{n}\ell_i\sum_{j=i}^n\ve_j,
\quad
(\nu,\ve_k)=\frac{\ell_1}{2}+\sum_{i=2}^{n}\ell_i\sum_{j=i}^n\dt_{j,k}
=\frac{\ell_1}{2}+\sum_{i=2}^{k}\ell_i, \quad k=1,\ldots, n.
$$
\begin{propn}
  For any quasi-classical finite dimensional module $V$, the extremal projector $p_\g=p_{\g/\l}p_\l$ is well defined
 on $1_V\tp Z^+_V$.
 \label{proj_defined_on VM}
\end{propn}
\begin{proof}
Denote by $W=\sum_{\xi\in \La(Z^+_V)} (V\tp Z)[\nu+\xi]$, the  sum of weight spaces in $V\tp Z$ of all weights of singular vectors. It contains $1_V\tp Z^+_V$ as a vector subspace.
We will show that all factors $p_\al(t)$, $\al \in \Rm^+$, are well defined at $t=(\rho,\al^\vee)$ on
$W$.
That is true for $\al\in \Rm^+_\l$ since $Z$ is locally finite over $U_q(\l)$. Moreover,
$p_\l(0) W$ is in  $U_q(\l_+)$-invariants since all weights of $W$ are
$\Rm_\l$-dominant (by virtue of  Proposition \ref{reg_loc_fin} for $\g=\l$). So we can further assume $\al \in \Rm^+_{\g/\l}$.

Present $\xi\in \La(Z^+_V)$ as $\xi=\la-\sum_{i=1}^{n}m_i\ve_i$ with $m_i\leqslant \ell_i$.
For $\al=\ve_i+\ve_j$ with $i<j$ we find
$$
 [(\nu+\xi+\rho,\al^\vee)]_{q_\al}=[\ell_1+\sum_{l=2}^{i}\ell_l+\sum_{l=2}^{j}\ell_l-m_i-m_j +i+j-2]_q.
$$
The integer in the square brackets in the  right-hand side is positive, hence $p_\al(\rho_\al)$ is well defined, by Proposition \ref{easy_case}.

For short roots $\al=\ve_i$, $i=1,\ldots,n$, the expression  $[(\nu+\xi+\rho,\al^\vee)+k]_{q_\al}$
does not turn zero at all $k\in \Z$ as
it is proportional to $q^{\frac{1}{2}+k'}+q^{-\frac{1}{2}-k'}$ for some integer $k'$ ($q$ is not a root of unity). So
the  series (\ref{translationed_proj}) for $p_\al(t)$ is regular at $t=(\rho,\al^\vee)$.
This also proves that the extremal projector $p_{\al_1}(\rho_1)$ is well defined on $W$.
Finally,  $p_\g(0)W$ is  annihilated by each $e_{\al}$  with $\al\in \Pi$ since one can choose a normal order with $\al$ on the left.
\end{proof}

As all weights of $Z$ are multiplicity free, we can write, up to a non-zero factor:
$$\theta_{Z,V} w\propto \prod_{\al\in \Rm_{\g/\l}^+}\prod_{k=1}^{l_{\xi,\al}}\frac{[(\nu+\rho+\xi,\al^\vee)+k]_{q_\al}}{[(\nu+\rho,\al^\vee)-k]_{q_\al}}w,
\quad w\in Z^+_V[\xi],
$$
where $l_{\xi,\al}=\max\{l \in \Z:e_\al^lw\not =0 \}$. This is a corollary of the formula
(\ref{proj_eigen}). In particular, for  $\xi=\la-\sum_{i=1}^{n}m_i\ve_i$  we have
$l_{\xi, \ve_{i}}=m_i$ and $l_{\xi, \ve_j+\ve_{i}}=\min(m_j,m_i)$,
where  $i\not =j$.
Introduce the shortcuts  $\phi_{\xi,\al,k}$ for $\frac{[(\nu+\rho+\xi,\al^\vee)+k]_{q_\al}}{[(\nu+\rho,\al^\vee)-k]_{q_\al}}$.
Then
$$
\det(\theta_{Z,V})\propto\prod_{\xi}\prod_{\al\in \Rm_{\g/\l}^+}\prod_{k=1}^{l_{\xi,\al}}\phi_{\xi,\al,k},
\quad \mbox{where} \quad
\xi\in \{\la-\sum_{i=1}^{n}m_i\ve_i\}_{m_i\leqslant \ell_i}.
$$
Note that factors corresponding to  roots $\al \in \Rm_\l^+$ are absent in the product because the operator
$p_\l(\nu)$ is invertible on $Z^+_V$ due to local finiteness of $Z$ with respect to $U_q(\l)$.

\begin{propn}
 The operator $\theta_{Z,V}$ is invertible.
\end{propn}
\begin{proof}
We should prove that  $\phi_{\xi,\al,k}\not =0$ for all $\al \in \Rm^+_{\g/\l}$.
For short $\al$, neither the denominator nor enumerator in $\phi_{\xi,\al,k}$ turn zero since
they are of the form $[(\la,\al^\vee)+k]_{q^\frac{1}{2}}$ with $k\in \Z$, cf. the proof of Proposition \ref{proj_defined_on VM}. So we have to check it only for $\al=\ve_i+\ve_j\in \Rm^+_{\k/\l}$, $i\not =j$.
Then
$$
\phi_{\xi,\al,k}=\frac{[\ell_1+\sum_{l=2}^{i}\ell_l+\sum_{l=2}^{j}\ell_l-m_i-m_j +i+j-2+k]_q}{[\ell_1+\sum_{l=2}^{i}\ell_l+\sum_{l=2}^{j}\ell_l+i+j-2-k]_q}
$$
 does not vanish since $k\leqslant l_{\xi,\al}=\min\{m_i,m_j\}\leqslant \min\{\ell_i,\ell_j\}.$
\end{proof}
\begin{corollary}
  For any quasi-classical finite dimensional $U_q(\g)$-module $V$, the tensor product $V\tp Z$ is completely reducible.
\end{corollary}
The irreducible components of $V\tp Z$ are pseudo-parabolic modules described in \cite{M2}.

\end{document}